\documentclass{article}

\usepackage[english]{babel}
\usepackage{graphicx,amsfonts,amssymb,amsmath,amsthm,enumitem,mathtools,accents,color}
\usepackage[top=2cm,bottom=3cm,left=2cm,right=2cm]{geometry}
\usepackage{sectsty}\sectionfont{\normalsize}\subsectionfont{\normalsize}

\newtheorem{theorem}{Theorem}[section]
\newtheorem{lemma}{Lemma}[section]
\theoremstyle{definition}
\newtheorem{definition}{Definition}[section]
\newtheorem{remark}{Remark}[section]
\newtheorem{example}{Example}[section]

\numberwithin{equation}{section}
\numberwithin{figure}{section}

\renewcommand{\epsilon}{\varepsilon}
\renewcommand{\aa}{{\boldsymbol a}}
\newcommand{\bb}{{\boldsymbol b}}
\newcommand{\ii}{{\boldsymbol i}}
\newcommand{\kk}{{\boldsymbol k}}
\newcommand{\hh}{{\boldsymbol h}}
\newcommand{\nn}{{\boldsymbol n}}
\newcommand{\xx}{{\boldsymbol x}}
\newcommand{\bu}{{\mathbf 1}}
\newcommand{\II}{{\mathcal I}}

\newcommand\rurldoi[1]{%
	\href{https://doi.org/#1}{\nolinkurl{#1}}%
}

\date{}
\begin{document}

\title{Constructive approach to the monotone rearrangement of functions\thanks{This is a preprint. For the peer reviewed and published version, see https://doi.org/10.1016/j.exmath.2021.10.004}}
\author{Giovanni Barbarino\\[-2pt]
\footnotesize Department of Mathematics and Systems Analysis, Aalto University, Espoo, Finland (giovanni.barbarino@aalto.fi)\\[7pt]
Davide Bianchi\\[-2pt]
\footnotesize School of Science, Harbin Institute of Technology (Shenzhen), Shenzhen, China (bianchi@hit.edu.cn)\\[7pt]
Carlo Garoni\\[-2pt]
\footnotesize Department of Mathematics, University of Rome Tor Vergata, Rome, Italy (garoni@mat.uniroma2.it)}

\maketitle

\begin{abstract}
We detail a simple procedure (easily convertible to an algorithm) for constructing from quasi-uniform samples of $f$ a sequence of linear spline functions converging to the monotone rearrangement of $f$, in the case where $f$ is an almost everywhere continuous function defined on a bounded set $\Omega$ with negligible boundary. 
Under additional assumptions on $f$ and $\Omega$, we prove that the convergence of the sequence is uniform.
We also show that the same procedure applies to arbitrary measurable functions too, but with the substantial difference that in this case the procedure has only a theoretical interest and cannot be converted to an algorithm.

\smallskip


\smallskip

\end{abstract}


\section{Introduction}

The theory of equimeasurable rearrangements of functions experienced its first major development soon after the seminal works by Hardy, Littlewood and P\'{o}lya \cite{HLP1,HLP2}.
Nowadays, the literature on this topic is incredibly vast; see, for example, 
\cite{Baernstein1,Baernstein2,quantile1,quantile3,Brenier,BS,Chiti,Chong-Rice,CRZ,Douglas1,Douglas2,Fristedt,Kawohl,Kesavan,Kolyada,Monti,Talenti1,Talenti2} 
for a look at some of the best contributions, without claiming completeness. We invite the interested reader to consult the remarkable 
review by Talenti~\cite{Talenti3} and the references therein.

In the last two decades, the theory of generalized locally Toeplitz (GLT) sequences, which originated from the pioneering papers by Tilli \cite{Tilli} and Serra-Capizzano \cite{glt_1,glt_2},
experienced a considerable development due to its numerous applications; see the recent books \cite{GLTbookIII,GLTbookIV,GLTbookI,GLTbookII}. The theory of GLT sequences is a powerful apparatus for computing the asymptotic spectral distribution of matrices $A_n$ arising from the discretization of linear differential equations. In particular, this theory allows for the computation of the so-called spectral symbol, i.e., the function describing the asymptotic spectrum of $A_n$ as the mesh fineness parameter $n$ diverges. It turns out that the monotone rearrangement of the symbol describes the asymptotic spectrum of $A_n$ even better than the symbol itself \cite{graphsI,bianchi,bianchi0}. For application purposes, it is then necessary to have a simple procedure for constructing the monotone rearrangement of a function from the function itself. 

In the main result of this paper (Theorem~\ref{th:main}), we detail a simple procedure (easily convertible to an algorithm) for constructing from quasi-uniform samples of $f$ a sequence of linear spline functions converging to the monotone rearrangement of $f$, in the case where $f$ is an almost everywhere continuous function defined on a bounded set $\Omega$ with negligible boundary. We also show that the convergence of the sequence is uniform as long as $f$ and $\Omega$ satisfy some additional assumptions (see Theorem~\ref{th:main'}). The construction presented in Theorem~\ref{th:main} has been largely employed within the theory of GLT sequences, although in a non-rigorous way and without a mathematical formalization; see \cite{graphsI,barbarinoDM,GLTbookIII,FE-DG,bianchi,bianchi0,MariaLucia,ALIF,axioms,GLTbookI,Tom-paper}. In this regard, Theorem~\ref{th:main} represents the first rigorous derivation (and generalization) of something that has been used so far in a heuristic way. 

After proving Theorems~\ref{th:main} and~\ref{th:main'} in Section~\ref{sec:main}, we discuss in Section~\ref{wagmf} the case of arbitrary measurable functions. For a function of this kind, a simple construction of the monotone rearrangement from quasi-uniform samples is clearly impossible, because the samples of the function may have nothing to do with the function itself. Nevertheless, we prove in Theorem~\ref{th:main2} that for any measurable function $f$ there exist ``good'' quasi-uniform samples of $f$ for which the same procedure as described in Theorem~\ref{th:main} yields the monotone rearrangement of $f$. Of course, Theorem~\ref{th:main2} has only a theoretical interest, 
since a recipe for finding these ``good'' quasi-uniform samples does not exist.

\section{Monotone rearrangement}\label{sec:mr}

In this section, we recall the notion of monotone rearrangement and collect some basic properties.
We denote by $C_c(\mathbb R)$ the space of continuous functions $F:\mathbb R\to\mathbb R$ with compact support, 
by $\chi_E$ the characteristic (indicator) function of the set $E$, and by $\mu_d$ the Lebesgue measure in $\mathbb R^d$. All the terminology coming from measure theory (such as ``measurable set'', ``measurable function'', ``a.e.'', etc.)\ always refers to the Lebesgue measure. Throughout this paper, we use a notation borrowed from probability theory to indicate sets. For example, if $f,g:E\subseteq\mathbb R^d\to\mathbb R$, then $\{f\le1\}=\{\xx\in E:f(\xx)\le1\}$, $\mu_d\{f>0,\:g<0\}$ is the measure of the set $\{\xx\in E:f(\xx)>0,\:g(\xx)<0\}$, etc. 

\begin{definition}
Let $f:\Omega\subset\mathbb R^d\to\mathbb R$ be measurable on a set $\Omega$ with $0<\mu_d(\Omega)<\infty$. The monotone rearrangement of $f$ is the function denoted by $f^\dag$ and defined as follows:
\begin{equation}\label{crv}
f^\dag:(0,1)\to\mathbb R,\qquad f^\dag(y)=\inf\biggl\{u\in\mathbb R:\frac{\mu_d\{f\le u\}}{\mu_d(\Omega)}\ge y\biggr\}.
\end{equation}
\end{definition}

Note that $f^\dag(y)$ is a well-defined real number for every $y\in(0,1)$, because
\[ \lim_{u\to+\infty}\mu_d\{f\le u\}=\mu_d(\Omega),\qquad\lim_{u\to-\infty}\mu_d\{f\le u\}=0. \]
In probability theory, where $f$ is interpreted as a random variable on $\Omega$ with probability distribution $m_f$ and distribution function $F_f$ given by
\begin{align*}
m_f(A)&=\mathbb P\{f\in A\}=\frac{\mu_d\{f\in A\}}{\mu_d(\Omega)},\qquad A\subseteq\mathbb R\mbox{ is a Borel set},\\
F_f(u)&=\mathbb P\{f\le u\}=\frac{\mu_d\{f\le u\}}{\mu_d(\Omega)},\qquad u\in\mathbb R,
\end{align*}
the monotone rearrangement $f^\dag$ in \eqref{crv} can be rewritten as
\[ f^\dag(y)=\inf\bigl\{u\in\mathbb R:F_f(u)\ge y\bigr\},\qquad y\in(0,1), \]
and is referred to as the quantile function of $f$ or the generalized inverse of $F_f$.
The next lemma collects some basic properties of monotone rearrangements; see, e.g., \cite[Chapter~3, Proposition~4 and Problem~3]{Fristedt} for the first two statements and \cite[Chapter~14, Proposition~7]{Fristedt} for the last one. 

\begin{lemma}\label{lemma_completo}
Let $f:\Omega\subset\mathbb R^d\to\mathbb R$ be measurable on a set $\Omega$ with $0<\mu_d(\Omega)<\infty$. 
\begin{itemize}[nolistsep,leftmargin=*]
	\item $f^\dag$ is monotone non-decreasing and left-continuous on $(0,1)$.
	\item For every Borel set $A\subseteq\mathbb R$, we have
	\[ \mu_1\{f^\dag\in A\}=\dfrac{\mu_d\{f\in A\}}{\mu_d(\Omega)}. \] 
	In particular,
	\begin{itemize}[nolistsep,leftmargin=*]
		\item if $f$ is bounded from below then $\inf_{y\in(0,1)}f^\dag(y)=\lim_{y\to0}f^\dag(y)\in\mathbb R$, 
		\item if $f$ is bounded from above then $\sup_{y\in(0,1)}f^\dag(y)=\lim_{y\to1}f^\dag(y)\in\mathbb R$.
	\end{itemize}
	\item For every continuous bounded function $F:\mathbb R\to\mathbb R$, we have
	\[ \frac1{\mu_d(\Omega)}\int_\Omega F(f(\xx)){\rm d}\xx=\int_0^1F(f^\dag(y)){\rm d}y. 
	\]
\end{itemize}
\end{lemma}


Lemma~\ref{g_cont} shows that the monotone rearrangement of a continuous function $f$ is continuous provided the domain of $f$ is not ``too wild''.
For the proof of Lemma~\ref{g_cont}, we need the following topological result. If $\xx\in\mathbb R^d$ and $r>0$, we denote by $D(\xx,r)$ the open $d$-dimensional disk with center $\xx$ and radius $r$. The closure and the interior of a set $\Omega\subseteq\mathbb R^d$ are denoted by $\overline\Omega$ and $\accentset{\circ}\Omega$, respectively.

\begin{lemma}\label{int_clos}
Let $\Omega\subseteq\mathbb R^d$ be a set contained in the closure of its interior.
Suppose that $U\subseteq\Omega$ is non-empty and open in $\Omega$.
Then $\mu_d(U)>0$.
\end{lemma}
\begin{proof}
Take a point $\xx\in U$. Since $U$ is open in $\Omega$, there exists a set $V$ open in $\mathbb R^d$ such that $U=V\cap\Omega$. In particular, we can find $\epsilon>0$ such that $D(\xx,\epsilon)\subseteq V$ and $D(\xx,\epsilon)\cap\Omega\subseteq U$. Now,
\begin{align*}
\xx\in U\subseteq\Omega\subseteq\overline{\accentset{\circ}\Omega}&\quad\implies\quad\xx\in\overline{\accentset{\circ}{\Omega}}\quad\implies\quad D(\xx,\epsilon)\cap\accentset{\circ}\Omega\ne\emptyset.
\end{align*}
We have then found a set $D(\xx,\epsilon)\cap\accentset{\circ}\Omega$, which is non-empty, open in $\mathbb R^d$, and contained in $U$. The measure of this set is clearly positive just like the measure of any non-empty open set in $\mathbb R^d$. We conclude that $\mu_d(U)>0$.
\end{proof}

\begin{lemma}\label{g_cont}
Let $f:\Omega\subset\mathbb R^d\to\mathbb R$ be continuous on a set $\Omega$ with $0<\mu_d(\Omega)<\infty$. Suppose that $\Omega$ is connected and contained in the closure of its interior. Then, the monotone rearrangement $f^\dag$ is continuous on $(0,1)$.
\end{lemma}
\begin{proof}

\begin{figure}
\centering
\includegraphics[width=0.40\textwidth]{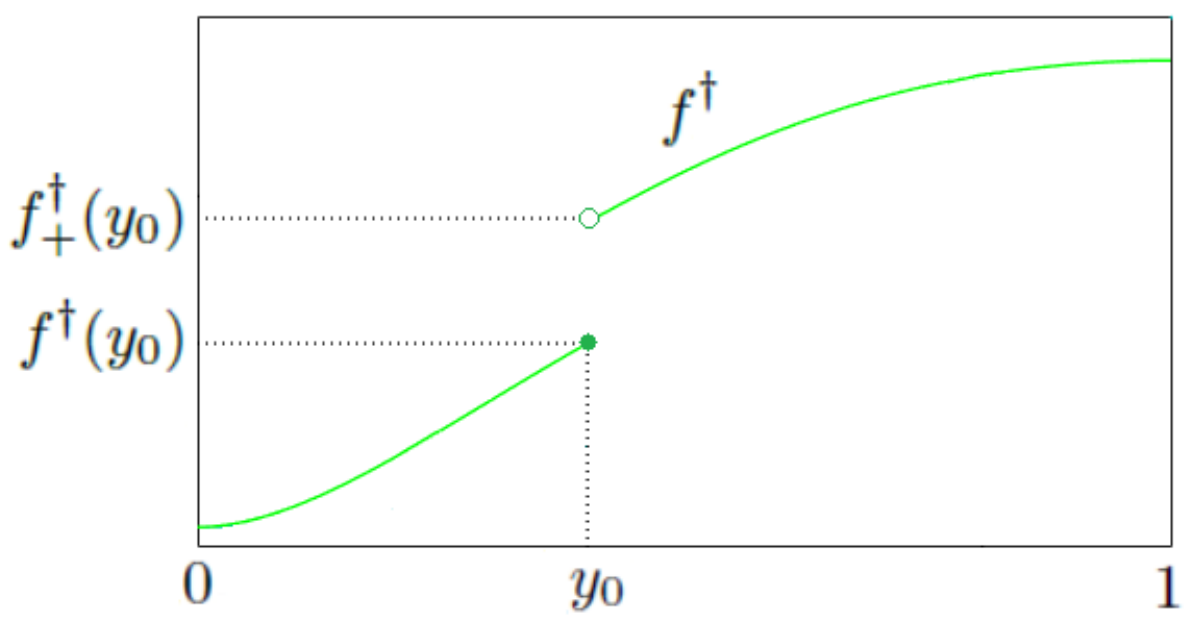}
\caption{Illustration for the proof of Lemma~\ref{g_cont}.}
\label{g_y0_discontinuo}
\end{figure}

Suppose by contradiction that $f^\dag$ is not continuous at a point $y_0\in(0,1)$. By monotonicity, $y_0$ is a jump discontinuity of $f^\dag$. Recalling that $f^\dag$ is left-continuous, we have
\[ f^\dag_-(y_0)=f^\dag(y_0)<f^\dag_+(y_0), \]
where $f^\dag_-(y_0)$ and $f^\dag_+(y_0)$ are, respectively, the left- and right-limit of $f^\dag$ in $y_0$; see Figure~\ref{g_y0_discontinuo}. By Lemma~\ref{lemma_completo},
\begin{alignat*}{5}
0&<y_0=\mu_1\{f^\dag\le f^\dag(y_0)\}=\frac{\mu_d\{f\le f^\dag(y_0)\}}{\mu_d(\Omega)} &&\quad\implies\quad \mu_d\{f\le f^\dag(y_0)\}>0,\\
0&<1-y_0=\mu_1\{f^\dag\ge f^\dag_+(y_0)\}=\frac{\mu_d\{f\ge f^\dag_+(y_0)\}}{\mu_d(\Omega)} &&\quad\implies\quad \mu_d\{f\ge f^\dag_+(y_0)\}>0,\\
0&=\mu_1\{f^\dag(y_0)<f^\dag<f^\dag_+(y_0)\}=\frac{\mu_d\{f^\dag(y_0)<f<f^\dag_+(y_0)\}}{\mu_d(\Omega)} &&\quad\implies\quad \mu_d\{f^\dag(y_0)<f<f^\dag_+(y_0)\}=0.
\end{alignat*}
Since $f$ is continuous on $\Omega$,
\begin{itemize}[nolistsep,leftmargin=*]
	\item $\{f\le f^\dag(y_0)\}=f^{-1}((-\infty,f^\dag(y_0)])$ is closed in $\Omega$,
	\item $\{f\ge f^\dag_+(y_0)\}=f^{-1}([f^\dag_+(y_0),\infty))$ is closed in $\Omega$,
	\item $\{f^\dag(y_0)<f<f^\dag_+(y_0)\}=f^{-1}((f^\dag(y_0),f^\dag_+(y_0)))$ is open in $\Omega$.
\end{itemize}
Moreover, it is clear that the sets $\{f\le f^\dag(y_0)\}$, $\{f\ge f^\dag_+(y_0)\}$, $\{f^\dag(y_0)<f<f^\dag_+(y_0)\}$ are pairwise disjoint and
\[ \{f\le f^\dag(y_0)\}\cup\{f\ge f^\dag_+(y_0)\}\cup\{f^\dag(y_0)<f<f^\dag_+(y_0)\}=\Omega. \]
Since $\{f\le f^\dag(y_0)\}$ and $\{f\ge f^\dag_+(y_0)\}$ are non-empty (because they have positive measure) and $\{f^\dag(y_0)<f<f^\dag_+(y_0)\}$ is empty (because of Lemma~\ref{int_clos}), we conclude that $\Omega$ is not connected (a contradiction). 
\end{proof}

\section{Constructive rearrangement of a.e.\ continuous functions}\label{sec:main}
In our main result (Theorem~\ref{th:main}), we detail a simple procedure for constructing from quasi-uniform samples of $f$ an easy-to-manage sequence of linear spline functions $f_\nn^\dag$ converging a.e.\ to the monotone rearrangement $f^\dag$, 
in the case where $f$ is an a.e.\ continuous function defined on a bounded set $\Omega$ with $\mu_d(\Omega)>\mu_d(\partial\Omega)=0$. Functions of this kind essentially exhaust both the class of functions that one may encounter in real-world applications and the class of functions for which a construction of the monotone rearrangement from the samples of the function makes sense; see also Section~\ref{wagmf}. After proving Theorem~\ref{th:main}, we show in Theorem~\ref{th:main'} that the convergence of $f_\nn^\dag$ to $f^\dag$ is uniform whenever $f$ is continuous and bounded on $\Omega$ and $\Omega$ is connected and contained in the closure of its interior (as in Lemma~\ref{g_cont}).
Before moving on, we introduce some necessary notations. 

\bigskip

\noindent\textbf{Multi-index notation.}
A multi-index $\ii$ of size $d$, also called a $d$-index, is a vector in $\mathbb Z^d$. 
$\mathbf0,\,\bu,\,\mathbf2,\,\ldots$ are the vectors of all zeros, all ones, all twos, $\ldots$ (their size will be clear from the context). For any vector $\nn\in\mathbb R^d$, we set $N(\nn)=\prod_{j=1}^dn_j$ and we write $\nn\to\infty$ to indicate that $\min(\nn)\to\infty$. 
If $\hh,\kk\in\mathbb R^d$, an inequality such as $\hh\le\kk$ means that $h_j\le k_j$ for all $j=1,\ldots,d$.
If $\hh,\kk$ are $d$-indices such that $\hh\le\kk$, the $d$-index range $\{\hh,\ldots,\kk\}$ is the set $\{\ii\in\mathbb Z^d:\hh\le\ii\le\kk\}$. We assume for this set the standard lexicographic ordering:
\[ \Bigl[\ \ldots\ \bigl[\ [\ (i_1,\ldots,i_d)\ ]_{i_d=h_d,\ldots,k_d}\ \bigr]_{i_{d-1}=h_{d-1},\ldots,k_{d-1}}\ \ldots\ \Bigr]_{i_1=h_1,\ldots,k_1}. \]
For instance, in the case $d=2$ the ordering is
\begin{align*}
&(h_1,h_2),\,(h_1,h_2+1),\,\ldots,\,(h_1,k_2),\,(h_1+1,h_2),\,(h_1+1,h_2+1),\,\ldots,\,(h_1+1,k_2),\\
&\ldots\,\ldots\,\ldots,\,(k_1,h_2),\,(k_1,h_2+1),\,\ldots,\,(k_1,k_2).
\end{align*}
When a $d$-index $\ii$ varies in a $d$-index range $\{\hh,\ldots,\kk\}$ (this is often written as $\ii=\hh,\ldots,\kk$), it is understood that $\ii$ varies from $\hh$ to $\kk$ following the lexicographic ordering. If $\hh,\kk$ are $d$-indices with $\hh\le\kk$, the notation $\sum_{\ii=\hh}^\kk$ indicates the summation over all $\ii$ in $\{\hh,\ldots,\kk\}$.
Operations involving $d$-indices (or general vectors with $d$ components) that have no meaning in the vector space $\mathbb R^d$ must always be interpreted in the componentwise sense. For instance, $\boldsymbol j\hh=(j_1h_1,\ldots,j_dh_d)$, $\ii/\boldsymbol j=(i_1/j_1,\ldots,i_d/j_d)$, etc.

\bigskip

\noindent\textbf{Asymptotically uniform grids.}
If $\aa,\bb\in\mathbb R^d$ with $\aa\le\bb$, we denote by $(\aa,\bb]$ the $d$-dimensional rectangle $(a_1,b_1]\times\cdots\times(a_d,b_d]$. Similar meanings have the notations for the open $d$-dimensional rectangle $(\aa,\bb)$ and the closed $d$-dimensional rectangle $[\aa,\bb]$.
Let $[\aa,\bb]$ be a $d$-dimensional rectangle, let $\nn\in\mathbb N^d$, and let $\mathcal G_\nn=\{\xx_{\ii,\nn}\}_{\ii=\bu,\ldots,\nn}$ be a sequence of $N(\nn)$ grid points in $\mathbb R^d$. We say that the grid $\mathcal G_\nn$ is asymptotically uniform (a.u.)\ in $[\aa,\bb]$ if
\[ \lim_{\nn\to\infty}\biggl(\max_{\ii=\bu,\ldots,\nn}\left\|\xx_{\ii,\nn}-\Bigl(\aa+\ii\,\frac{\bb-\aa}\nn\Bigr)\right\|_\infty\biggr)=0, \]
where $\|\xx\|_\infty=\max(|x_1|,\ldots,|x_d|)$ for every $\xx\in\mathbb R^d$.

\bigskip

\noindent\textbf{Standard partitions.} Let $[\aa,\bb]$ be a $d$-dimensional rectangle and let $\nn\in\mathbb N^d$. We say that $\{I_{\ii,\nn}\}_{\ii=\bu,\ldots,\nn}$ is a standard partition of $[\aa,\bb]$ if it is a partition of $[\aa,\bb]$ such that
\[ \biggl(\aa+(\ii-\bu)\frac{\bb-\aa}{\nn},\aa+\ii\frac{\bb-\aa}{\nn}\biggr)\subseteq I_{\ii,\nn}\subseteq\biggl[\aa+(\ii-\bu)\frac{\bb-\aa}{\nn},\aa+\ii\frac{\bb-\aa}{\nn}\biggr],\qquad \ii=\bu,\ldots,\nn. \]
If $\{I_{\ii,\nn}\}_{\ii=\bu,\ldots,\nn}$ is a standard partition of $[\aa,\bb]$ and $\xx_{\ii,\nn}\in I_{\ii,\nn}$ for all $\ii=\bu,\ldots,\nn$, then $\mathcal G_\nn=\{\xx_{\ii,\nn}\}_{\ii=\bu,\ldots,\nn}$ is an a.u.\ grid in $[\aa,\bb]$.

\bigskip

\noindent\textbf{Regular sets.}
We say that $\Omega\subset\mathbb R^d$ is a regular set if it is bounded and $\mu_d(\partial\Omega)=0$. Note that the condition ``$\mu_d(\partial\Omega)=0$'' is equivalent to ``$\chi_\Omega$ is continuous a.e.\ on $\mathbb R^d$''. Any regular set $\Omega\subset\mathbb R^d$ is measurable and we have $\mu_d(\Omega)=\mu_d(\accentset{\circ}\Omega)=\mu_d(\overline\Omega)<\infty$. 

\bigskip

\noindent\textbf{Statement and proof of the main result.}
With the notations introduced in the previous paragraphs, we are now ready to state our main result.

\begin{theorem}\label{th:main}
Let $f:\Omega\subset\mathbb R^d\to\mathbb R$ be continuous a.e.\ on the regular set $\Omega$ with $\mu_d(\Omega)>0$. 
Take any $d$-dimensional rectangle $[\aa,\bb]$ containing $\Omega$ and any a.u.\ grid $\mathcal G_\nn=\{\xx_{\ii,\nn}\}_{\ii=\bu,\ldots,\nn}$ in $[\aa,\bb]$. 
For each $\nn\in\mathbb N^d$, consider the samples
\begin{equation*}
f(\xx_{\ii,\nn}),\qquad\ii\in\II_\nn(\Omega)=\{\ii\in\{\bu,\ldots,\nn\}:\xx_{\ii,\nn}\in\Omega\},
\end{equation*}
sort them in non-decreasing order, and put them into a vector $(s_0,\ldots,s_{\omega(\nn)})$, where $\omega(\nn)=\#\II_\nn(\Omega)-1$.
Let $f^\dag_\nn:[0,1]\to\mathbb R$ be the linear spline function that interpolates the samples $(s_0,\ldots,s_{\omega(\nn)})$ over the equally spaced nodes $(0,\frac{1}{\omega(\nn)},\frac{2}{\omega(\nn)},\ldots,1)$ in $[0,1]$. Then,
\[ \lim_{\nn\to\infty} f^\dag_\nn(y)=f^\dag(y) \] 
for every continuity point $y$ of $f^\dag$. In particular, $f^\dag_\nn\to f^\dag$ a.e.\ in $(0,1)$.
\end{theorem}

Note that the last 
statement in Theorem~\ref{th:main} follows from the observation that $f^\dag$ is monotone and hence the set of its discontinuity points is countable \cite[Theorem~4.30]{Rudinino}. 
For the proof of Theorem~\ref{th:main}, we need three lemmas, which have an interest also in themselves. The first two provide a generalization of \cite[Lemma~3.1]{a.u.grids}. The third one is an improved version of a result never published \cite[Lemma~4.2]{barbarinoDM}.
We say that a function $f:\mathbb R^d\to\mathbb R$ is locally bounded on $\mathbb R^d$ if it is bounded on every compact subset of $\mathbb R^d$.

\begin{lemma}\label{illemma}
Let $f:\mathbb R^d\to\mathbb R$ be measurable and let $[\aa,\bb]$ be a $d$-dimensional rectangle. 
Consider a standard partition $\{I_{\ii,\nn}\}_{\ii=\bu,\ldots,\nn}$ of $[\aa,\bb]$ and 
let $\mathcal G_\nn=\{\xx_{\ii,\nn}\}_{\ii=\bu,\ldots,\nn}$ be a.u.\ in $[\aa,\bb]$. Then,
\begin{equation}\label{lim_c}
\lim_{\nn\to\infty}\sum_{\ii=\bu}^\nn f(\xx_{\ii,\nn})\chi_{I_{\ii,\nn}}(\xx)=f(\xx)
\end{equation}
for every $\xx\in[\aa,\bb]$ which is a continuity point of $f$.
In particular, if $f$ is continuous a.e.\ and locally bounded on $\mathbb R^d$, then
\begin{equation}\label{a.e.c}
\lim_{\nn\to\infty}\sum_{\ii=\bu}^\nn f(\xx_{\ii,\nn})\chi_{I_{\ii,\nn}}(\xx)=f(\xx)\mbox{ \,for a.e.\ $\xx\in[\aa,\bb]$}
\end{equation}
and
\begin{equation}\label{Rsum}
\lim_{\nn\to\infty}\frac1{N(\nn)}\sum_{\ii=\bu}^\nn f(\xx_{\ii,\nn})=\frac1{\mu_d([\aa,\bb])}\int_{[\aa,\bb]}f(\xx){\rm d}\xx.
\end{equation}
\end{lemma}
\begin{proof}
Let $S=[\aa-\bu,\bb+\bu]$. The grid $\mathcal G_\nn$ is a.u.\ in $[\aa,\bb]$ and hence it is eventually contained in $S$ as $\nn\to\infty$. Without loss of generality, we can assume that $\mathcal G_\nn\subset S$ for all $\nn$.
Let $\xx\in[\aa,\bb]$ be a continuity point of $f$ and fix $\epsilon>0$. Then, there is a $\delta=\delta_{\xx,\epsilon}>0$ such that $|f(\boldsymbol y)-f(\xx)|\le\epsilon$ whenever $\boldsymbol y\in S$ and $\|\boldsymbol y-\xx\|_\infty\le\delta$. Since $\mathcal G_\nn$ is a.u.\ in $[\aa,\bb]$, we can choose $\nn_\delta$ such that, for $\nn\ge\nn_\delta$,
\[ \max_{\ii=\bu,\ldots,\nn}\left\|\xx_{\ii,\nn}-\Bigl(\aa+\ii\frac{\bb-\aa}\nn\Bigr)\right\|_\infty\le\frac\delta2,\qquad\biggl\|\frac{\bb-\aa}{\nn}\biggr\|_\infty\le\frac\delta2. \]
For $\nn\ge\nn_\delta$, if we call $I_{\kk,\nn}$ the unique element of the standard partition containing $\xx$, we have
\[ \|\xx_{\kk,\nn}-\xx\|_\infty\le\left\|\xx_{\kk,\nn}-\Bigl(\aa+\kk\frac{\bb-\aa}\nn\Bigr)\right\|_\infty+\left\|\Bigl(\aa+\kk\frac{\bb-\aa}\nn\Bigr)-\xx\right\|_\infty\le\frac\delta2+\frac\delta2=\delta \]
and
\[ \left|\sum_{\ii=\bu}^\nn f(\xx_{\ii,\nn})\chi_{I_{\ii,\nn}}(\xx)-f(\xx)\right|=|f(\xx_{\kk,\nn})-f(\xx)|\le\epsilon. \]
As a consequence, $\sum_{\ii=\bu}^\nn f(\xx_{\ii,\nn})\chi_{I_{\ii,\nn}}(\xx)\to f(\xx)$ whenever $\xx\in[\aa,\bb]$ is a continuity point of $f$. This proves \eqref{lim_c}.
In the case where $f$ is continuous a.e.\ and locally bounded on $\mathbb R^d$, the limit \eqref{a.e.c} follows immediately from \eqref{lim_c}, 
while the limit \eqref{Rsum} follows from \eqref{a.e.c} and the dominated convergence theorem, taking into account that
\[ \left|\sum_{\ii=\bu}^\nn f(\xx_{\ii,\nn})\chi_{I_{\ii,\nn}}\right|\le\|f\|_{\infty,S}<\infty,\qquad\int_{[\aa,\bb]}\Biggl(\sum_{\ii=\bu}^\nn f(\xx_{\ii,\nn})\chi_{I_{\ii,\nn}}\Biggr)=\frac{N(\bb-\aa)}{N(\nn)}\sum_{\ii=\bu}^\nn f(\xx_{\ii,\nn}), \]
and $N(\bb-\aa)=\mu_d([\aa,\bb])$. 
\end{proof}

\begin{remark}\label{chi_Omega}
Let $\Omega\subset\mathbb R^d$ be a regular set contained in the $d$-dimensional rectangle $[\aa,\bb]$, let $\mathcal G_\nn=\{\xx_{\ii,\nn}\}_{\ii=\bu,\ldots,\nn}$ be a.u.\ in $[\aa,\bb]$, and let $\II_\nn(\Omega)=\{\ii\in\{\bu,\ldots,\nn\}:\xx_{\ii,\nn}\in\Omega\}$. By applying the limit \eqref{Rsum} in Lemma~\ref{illemma} with $f=\chi_\Omega$, we obtain
\[ \lim_{\nn\to\infty}\frac{\#\II_\nn(\Omega)}{N(\nn)}=\frac{\mu_d(\Omega)}{\mu_d([\aa,\bb])} .\]
\end{remark}

\begin{lemma}\label{ilcor}
Let $f:\Omega\subset\mathbb R^d\to\mathbb R$ be measurable on the bounded set $\Omega$ and 
let $[\aa,\bb]$ be a $d$-dimensional rectangle containing $\Omega$. Consider a standard partition $\{I_{\ii,\nn}\}_{\ii=\bu,\ldots,\nn}$ of $[\aa,\bb]$, 
let $\mathcal G_\nn=\{\xx_{\ii,\nn}\}_{\ii=\bu,\ldots,\nn}$ be a.u.\ in $[\aa,\bb]$, and set $\II_\nn(\Omega)=\{\ii\in\{\bu,\ldots,\nn\}:\xx_{\ii,\nn}\in\Omega\}$. Then,
\begin{equation}\label{lim_c.O}
\lim_{\nn\to\infty}\sum_{\ii\in\II_\nn(\Omega)}f(\xx_{\ii,\nn})\chi_{I_{\ii,\nn}}(\xx)=f(\xx)
\end{equation}
for every $\xx\in\accentset{\circ}\Omega$ which is a continuity point of $f$.
In particular, if $\Omega$ is a regular set with $\mu_d(\Omega)>0$ and $f$ is continuous a.e.\ and bounded on $\Omega$,
then
\begin{equation}\label{a.e.c.O}
\lim_{\nn\to\infty}\sum_{\ii\in\II_\nn(\Omega)}f(\xx_{\ii,\nn})\chi_{I_{\ii,\nn}}(\xx)=f(\xx)\mbox{ \,for a.e.\ $\xx\in\Omega$}
\end{equation}
and
\begin{equation}\label{Rsum.O}
\lim_{\nn\to\infty}\frac1{\#\II_\nn(\Omega)}\sum_{\ii\in\II_\nn(\Omega)}f(\xx_{\ii,\nn})=\frac1{\mu_d(\Omega)}\int_{\Omega}f(\xx){\rm d}\xx.
\end{equation}
\end{lemma}
\begin{proof}
Let $\tilde f:\mathbb R^d\to\mathbb R$ be the extension of $f$ to $0$ outside $\Omega$.
By Lemma~\ref{illemma},
\[ \lim_{\nn\to\infty}\sum_{\ii=\bu}^\nn\tilde f(\xx_{\ii,\nn})\chi_{I_{\ii,\nn}}(\xx)=\tilde f(\xx) \]
for every $\xx\in[\aa,\bb]$ which is a continuity point of $\tilde f$. This implies that
\[ \lim_{\nn\to\infty}\sum_{\ii\in\II_\nn(\Omega)}f(\xx_{\ii,\nn})\chi_{I_{\ii,\nn}}(\xx)=f(\xx) \]
for every $\xx\in\accentset{\circ}\Omega$ which is a continuity point of $f$, and \eqref{lim_c.O} is proved. 
In the case where $\Omega$ is a regular set with $\mu_d(\Omega)>0$ and $f$ is continuous a.e.\ and bounded on $\Omega$,
the limit \eqref{a.e.c.O} follows from \eqref{lim_c.O} and the equation $\mu_d(\partial\Omega)=0$, while
the limit \eqref{Rsum.O} follows from \eqref{a.e.c.O} and the dominated convergence theorem,
taking into account that
\begin{align*}
&\left|\sum_{\ii\in\II_\nn(\Omega)}f(\xx_{\ii,\nn})\chi_{I_{\ii,\nn}}\right|\le\|f\|_{\infty,\Omega}<\infty,\\
&\int_\Omega\Biggl(\sum_{\ii\in\II_\nn(\Omega)}f(\xx_{\ii,\nn})\chi_{I_{\ii,\nn}}\Biggr)=\frac{N(\bb-\aa)}{N(\nn)}\sum_{\ii\in\II_\nn(\Omega)}f(\xx_{\ii,\nn})=\frac{\mu_d([\aa,\bb])\cdot\#\II_\nn(\Omega)}{N(\nn)}\cdot\frac1{\#\II_\nn(\Omega)}\sum_{\ii\in\II_\nn(\Omega)}f(\xx_{\ii,\nn}),\\
&\lim_{\nn\to\infty}\frac{\#\II_\nn(\Omega)}{N(\nn)}=\frac{\mu_d(\Omega)}{\mu_d([\aa,\bb])},
\end{align*}
where the latter is a consequence of Remark~\ref{chi_Omega}.
\end{proof}

\begin{lemma}\label{ultimo}
Let $\omega_m$ be a sequence of positive integers such that $\omega_m\to\infty$ and let $g_m:[0,1]\to\mathbb R$ be a sequence of non-decreasing functions such that
\[ \lim_{m\to\infty}\frac1{\omega_m}\sum_{\ell=0}^{\omega_m}F\Bigl(g_m\Bigl(\frac\ell{\omega_m}\Bigr)\Bigr)=\int_0^1F(g(y)){\rm d}y,\qquad\forall\,F\in C_c(\mathbb R), \]
where $g:(0,1)\to\mathbb R$ is non-decreasing. Then, $g_m(y)\to g(y)$ for every continuity point $y$ of $g$. 
\end{lemma}
\begin{proof}
Suppose by contradiction that there exists $\alpha\in(0,1)$ such that $\alpha$ is a continuity point of $g$ and $g_m(\alpha)\not\to g$. We then have $|g_m(\alpha)-g(\alpha)|>\epsilon$ infinitely often (i.o.)\ for some fixed $\epsilon>0$. There are two possible (non-mutually exclusive) cases.

\medskip

\noindent{\em Case 1: There exist infinite indices $m$ such that $g_m(\alpha)-g(\alpha)>\epsilon$.} Passing to a subsequence (if necessary), we can assume that
\begin{equation}\label{a}
g_m(\alpha)>g(\alpha)+\epsilon
\end{equation}
for all $m$. Since $\alpha$ is a continuity point of $g$, there exists $\delta>0$ such that
\begin{equation}\label{b}
g(\alpha+\delta)<g(\alpha)+\frac\epsilon2;
\end{equation}
see Figure~\ref{ggm}. Let $0<\gamma<\delta$ and take $F\in C_c(\mathbb R)$ such that $F=1$ on $[g(\gamma),g(\alpha+\delta)]$, $F=0$ on $[g(\alpha+\delta)+\frac\epsilon2,\infty)$ and $0\le F\le1$ on $\mathbb R$. Since $g(y)\in[g(\gamma),g(\alpha+\delta)]$ for $y\in[\gamma,\alpha+\delta]$ by the monotonicity of $g$, we have
\begin{align*}
\int_0^1F(g(y)){\rm d}y\ge\alpha+\delta-\gamma.
\end{align*}
Since $g_m(y)\ge g_m(\alpha)>g(\alpha)+\epsilon>g(\alpha+\delta)+\frac{\epsilon}2$ for $y\ge\alpha$ by \eqref{a}--\eqref{b} and the monotonicity of $g_m$, we have
\begin{align*}
\frac1{\omega_m}\sum_{\ell=0}^{\omega_m}F\Bigl(g_m\Bigl(\frac\ell{\omega_m}\Bigr)\Bigr)\le\frac{\#\{\ell\in\{0,\ldots,\omega_m\}:\frac{\ell}{\omega_m}\in[0,\alpha)\}}{\omega_m}\le\alpha+\frac1{\omega_m}.
\end{align*}
Using the hypothesis, we finally obtain
\[ \alpha\ge\lim_{m\to\infty}\frac1{\omega_m}\sum_{\ell=0}^{\omega_m}F\Bigl(g_m\Bigl(\frac\ell{\omega_m}\Bigr)\Bigr)=\int_0^1F(g(y)){\rm d}y\ge\alpha+\delta-\gamma, \]
which is a contradiction as $\gamma<\delta$.

\begin{figure}
\centering
\includegraphics[width=0.40\textwidth]{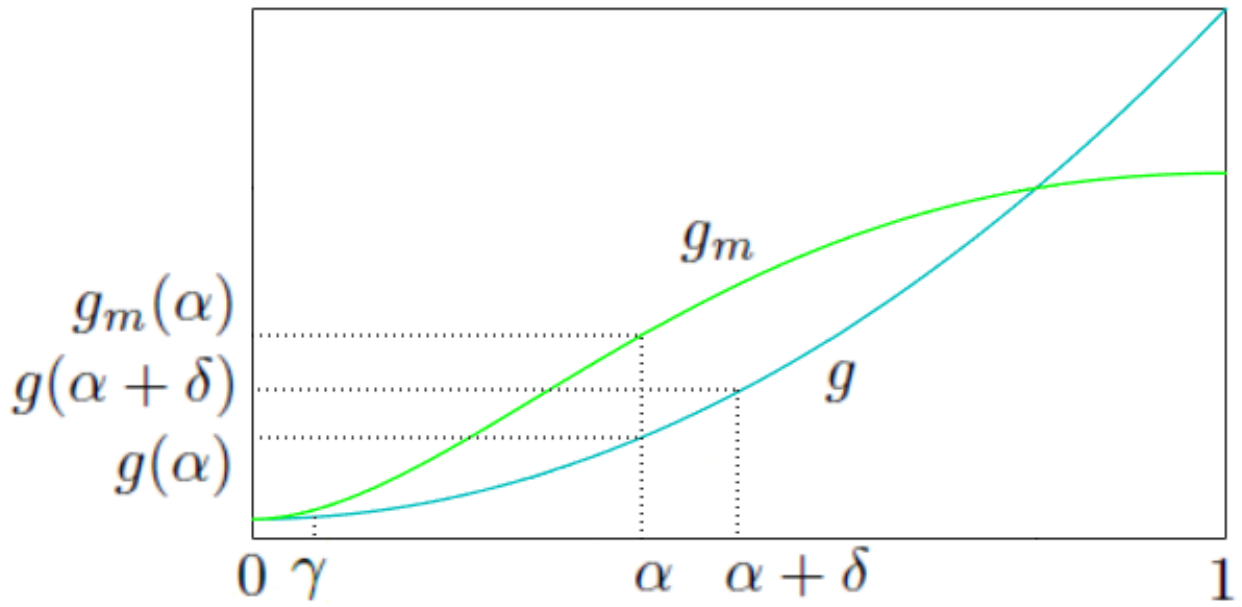}
\caption{Illustration for the proof of Lemma~\ref{ultimo}.}
\label{ggm}
\end{figure}

\medskip

\noindent{\em Case 2: There exist infinite indices $m$ such that $g_m(\alpha)-g(\alpha)<-\epsilon$.} 
Let $h(x)=-g(1-x)$ and $h_m=-g_m(1-x)$. These functions are non-decreasing and, by assumption,
\[\lim_{m\to\infty}\frac1{\omega_m}\sum_{\ell=0}^{\omega_m}F\Bigl(h_m\Bigl(\frac\ell{\omega_m}\Bigr)\Bigr)=
\lim_{m\to\infty}\frac1{\omega_m}\sum_{\ell=0}^{\omega_m}F\Bigl(-g_m\Bigl(\frac\ell{\omega_m}\Bigr)\Bigr)=\int_0^1F(-g(y)){\rm d}y
=\int_0^1F(h(y)){\rm d}y,\qquad\forall\,F\in C_c(\mathbb R).
\]
Moreover, the condition $g_m(\alpha)-g(\alpha)<-\epsilon$ can be rewritten as $h_m(1-\alpha)-h(1-\alpha)>\epsilon$, where $1-\alpha$ is a continuity point of $h$. We can therefore use the same argument as in Case~1 
and get a contradiction.
\end{proof}


\begin{proof}[Proof of Theorem~{\rm\ref{th:main}}]
For every $F\in C_c(\mathbb R)$, we have
\begin{align*}
\lim_{\nn\to\infty}\frac1{\omega(\nn)}\sum_{\ell=0}^{\omega(\nn)}F\Bigl(f^\dag_\nn\Bigl(\frac\ell{\omega(\nn)}\Bigr)\Bigr)&=\lim_{\nn\to\infty}\frac1{\omega(\nn)}\sum_{\ell=0}^{\omega(\nn)}F(s_\ell)=\lim_{\nn\to\infty}\frac1{\#\II_\nn(\Omega)-1}\sum_{\ii\in\II_\nn(\Omega)}F(f(\xx_{\ii,\nn}))\\
&=\frac1{\mu_d(\Omega)}\int_\Omega F(f(\xx)){\rm d}\xx=\int_0^1F(f^\dag(y)){\rm d}y,
\end{align*}
where the second-to-last equality follows from Lemma~\ref{ilcor} applied to the composite function $F(f)$---which is continuous a.e.\ and bounded on the regular set $\Omega$---while the last equality follows from Lemma~\ref{lemma_completo}. 
Since $f^\dag_\nn:[0,1]\to\mathbb R$ and $f^\dag:(0,1)\to\mathbb R$ are non-decreasing functions, and
\[ \omega(\nn)=\#\II_\nn(\Omega)-1\sim\frac{\mu_d(\Omega)N(\nn)}{\mu_d([\aa,\bb])}\to\infty\mbox{ \,as }\,\nn\to\infty \]
by Remark~\ref{chi_Omega}, we conclude by Lemma~\ref{ultimo} that $f^\dag_\nn(y)\to f^\dag(y)$ for every continuity point $y$ of $f^\dag$.
\end{proof}

We remark that Theorem~\ref{th:main} is clearly constructive. In particular, it can be easily converted to an algorithm that produces a sequence of linear spline functions $f^\dag_\nn$ converging to the monotone rearrangement $f^\dag$ of an a.e.\ continuous function $f$ defined on a regular set $\Omega$. Different versions of this algorithm, depending on the considered a.u.\ grid $\mathcal G_\nn$, have been largely used within the theory of GLT sequences \cite{graphsI,barbarinoDM,GLTbookIII,FE-DG,bianchi,bianchi0,MariaLucia,ALIF,axioms,GLTbookI,Tom-paper}. However, all these works have never provided a mathematical formalization of their construction, which is in fact the original contribution of this paper.

\bigskip

\noindent\textbf{Uniform convergence.} In Theorem~\ref{th:main'} we show that, under additional assumptions on $f$ and $\Omega$, the convergence $f^\dag_\nn\to f^\dag$ in Theorem~\ref{th:main} is uniform. The same result can also be derived from the main theorems in \cite{quantile3}, with the only difference that the authors of \cite{quantile3} adopt a probabilistic perspective, while here we propose a pure analytical approach. In Theorem~\ref{th:main'}, it is implicitly assumed that, if $f$ is bounded from below [above], then $f^\dag$ is defined also in $y=0$ [$y=1$] in the obvious way $f^\dag(0)=\lim_{y\to0}f^\dag(y)$ [$f^\dag(1)=\lim_{y\to1}f^\dag(y)$]; see also Lemma~\ref{lemma_completo}.

\begin{theorem}\label{th:main'}
In Theorem~{\rm\ref{th:main}}, suppose that $\Omega$ is connected and contained in the closure of its interior. Then, the following properties hold.
\begin{enumerate}[nolistsep,leftmargin=*]
	\item If $f$ is continuous on $\Omega$, 
	then $f_\nn^\dag\to f^\dag$ 
	uniformly on every compact interval $[\alpha,\beta]\subset(0,1)$.
	\item If $f$ is continuous and bounded from below on $\Omega$, then 
	$f_\nn^\dag\to f^\dag$ uniformly on every compact interval $[0,\alpha]\subset[0,1)$.
	\item If $f$ is continuous and bounded from above on $\Omega$, 
	then $f_\nn^\dag\to f^\dag$ uniformly on every compact interval $[\alpha,1]\subset(0,1]$.
	\item If $f$ is continuous and bounded on $\Omega$, 
	then $f_\nn^\dag\to f^\dag$ uniformly on $[0,1]$.
\end{enumerate}
\end{theorem}

The proof of Theorem~\ref{th:main'} relies on the following lemma, which is sometimes referred to as Dini's second theorem \cite[pp.~81 and 270, Problem~127]{Polya-Szego}.

\begin{lemma}\label{Dini}
If a sequence of monotone functions converges pointwise on a compact interval to a continuous function, then it converges uniformly.
\end{lemma}

\begin{proof}[Proof of Theorem~{\rm\ref{th:main'}}]\hfill
\begin{enumerate}[nolistsep,leftmargin=*]
	\item $f^\dag$ is continuous on $(0,1)$ by Lemma~\ref{g_cont} and $f_\nn^\dag\to f^\dag$ everywhere in $(0,1)$ by Theorem~\ref{th:main}. Since the functions $f_\nn^\dag$ are continuous and non-decreasing on $(0,1)$, the thesis follows from Lemma~\ref{Dini}.
	\item Since $f$ is bounded from below on $\Omega$, we have $\inf_{\xx\in\Omega}f(\xx)=m\in\mathbb R$. For every $\epsilon>0$, the set $\{f<m+\epsilon\}$ is non-empty (by definition of~$m$) and open in $\Omega$ (because $f$ is continuous on $\Omega$),
	and so it has non-zero measure by Lemma~\ref{int_clos}. Hence, for every $\epsilon>0$,
\begin{align*}
&\frac{\mu_d\{f\le m+\epsilon\}}{\mu_d(\Omega)}=z_\epsilon>0,\qquad\frac{\mu_d\{f\le m-\epsilon\}}{\mu_d(\Omega)}=0\qquad\implies\qquad m+\epsilon \ge f^\dag(z_\epsilon)\ge  \inf_{y\in (0,1)} f^\dag(y)\ge m-\epsilon,
\end{align*}
and it follows that
\[ m =  \inf_{y\in (0,1)} f^\dag(y)=\lim_{y\to0}f^\dag(y)=f^\dag(0). \]
Now, the function $f^\dag$ is continuous on $[0,1)$ by Lemma~\ref{g_cont} and the definition $f^\dag(0)=\lim_{y\to0}f^\dag(y)$.
Since $f_\nn^\dag(0)=s_0$ is the evaluation of $f$ in a point of $\Omega$, for every $\nn$ we have 
\[ f_\nn^\dag(0) \ge \inf_{\xx\in\Omega}f(\xx)=m=f^\dag(0). \] 
Since $f^\dag_\nn\to f^\dag$ everywhere in $(0,1)$ by Theorem~\ref{th:main} and the functions $f_\nn^\dag$, $f^\dag$ are continuous and non-decreasing on $[0,1)$, for every $\epsilon>0$ we have
\[ f^\dag(0)\le \liminf_{\nn\to \infty} f_\nn^\dag(0) \le \limsup_{\nn\to \infty} f_\nn^\dag(0) \le \limsup_{\nn\to \infty} f_\nn^\dag(\epsilon) = f^\dag(\epsilon), \]
hence
\[ f^\dag(0) = \lim_{\nn\to \infty} f_\nn^\dag(0). \]
We have therefore proved that $f_\nn^\dag\to f^\dag$ everywhere in $[0,1)$, and the thesis now follows from Lemma~\ref{Dini}. 
	\item It is proved in the same way as the second statement.
	\item It follows immediately from the second and third statements. \qedhere
\end{enumerate}
\end{proof}

\section{The case of general measurable functions}\label{wagmf}

Sampling a function that is not continuous a.e.\ does not make sense in general. In particular, we cannot expect to obtain an easy-to-manage sequence converging to the monotone rearrangement of an arbitrary measurable function through a simple sampling procedure as the one described in Theorem~\ref{th:main}. This is due to the fact that quasi-uniform samples of an arbitrary measurable function may have nothing to do with the function itself.

\begin{example}\label{Dir_fun}
Consider the Dirichlet function $f:[0,1]\to\mathbb R$,
\[ f(x)=\left\{\begin{aligned}
&1, &&\textup{if }x\in[0,1]\cap\mathbb Q,\\
&0, &&\textup{otherwise}.
\end{aligned}\right. \]
The function $f$ is measurable and $f=0$ a.e.\ in $[0,1]$. The monotone rearrangement $f^\dag$ is the identically zero function. If we follow the construction of Theorem~\ref{th:main} using any a.u.\ grid in $[0,1]$ consisting of rational numbers (e.g., $\mathcal G_n=\{x_{i,n}\}_{i=1,\ldots,n}$ with $x_{i,n}=\frac in$), we obtain a sequence of functions $f^\dag_n$ which are identically equal to~$1$ for all~$n$. Hence, there is no point $x\in(0,1)$ such that $f^\dag_n(x)\to f^\dag(x)$. We remark that real-world computations are performed by computers and every possible a.u.\ grid that a computer can use consists of rational numbers.
\end{example}

In Theorem~\ref{th:main2}, we show that, for a general measurable function $f$, there exist ``good'' quasi-uniform samples of $f$ that allow one to obtain the monotone rearrangement of $f$ 
by the same construction as in Theorem~\ref{th:main}. Unfortunately, Theorem~\ref{th:main2} has only a theoretical interest, because in general there is no way to select ``good'' quasi-uniform samples for which the whole construction works. After all, it could not be otherwise, considering that a general measurable function $f$ can be modified on a set of zero measure without changing (the essence of) $f$ but with a tremendous impact on the pointwise evaluations of $f$; see also Example~\ref{Dir_fun}.


\begin{theorem}\label{th:main2}
Let $f:\Omega\subset\mathbb R^d\to\mathbb R$ be measurable on a regular set $\Omega$ with $\mu_d(\Omega)>0$. 
Take any $d$-dimensional rectangle $[\aa,\bb]$ containing $\Omega$. Then, there exists an a.u.\ grid $\mathcal G_\nn=\{\xx_{\ii,\nn}\}_{\ii=\bu,\ldots,\nn}$ in $[\aa,\bb]$ for which the construction of Theorem~{\rm\ref{th:main}} works, in the following sense.
\begin{quote}
For each $\nn\in\mathbb N^d$, consider the samples
\begin{equation*}
f(\xx_{\ii,\nn}),\qquad\ii\in\II_\nn(\Omega)=\{\ii\in\{\bu,\ldots,\nn\}:\xx_{\ii,\nn}\in\Omega\},
\end{equation*}
sort them in non-decreasing order, and put them into a vector $(s_0,\ldots,s_{\omega(\nn)})$, where $\omega(\nn)=\#\II_\nn(\Omega)-1$.
Let $f^\dag_\nn:[0,1]\to\mathbb R$ be the linear spline function that interpolates the samples $(s_0,\ldots,s_{\omega(\nn)})$ over the equally spaced nodes $(0,\frac{1}{\omega(\nn)},\frac{2}{\omega(\nn)},\ldots,1)$ in $[0,1]$. Then, $f^\dag_\nn\to f^\dag$ a.e.\ in $(0,1)$ as $\nn\to\infty$.
\end{quote}
\end{theorem}

For the proof of Theorem~\ref{th:main2}, we need two lemmas. The first one is just a technical result \cite[Lemma~4.1]{GLTbookII}. The second one, which has an interest also in itself, is a fine-tuned version of \cite[Propositions~1.1 and~4.3(iv)]{Chong-Rice} and \cite[Chapter~14, Propositions~3 and~5]{Fristedt}, 
with the only difference that the sequence $f_m$ considered here has a domain $\Omega_m$ that varies with the sequence index $m$. In what follows, if $\Omega,\Omega_m\subseteq\mathbb R^d$ are measurable sets, we say that $\Omega_m\to\Omega$ if the following conditions are satisfied.
\begin{itemize}[nolistsep,leftmargin=*]
	\item $\mu_d\{\xx\in\Omega:\xx\not\in\Omega_m\mbox{ i.o.}\}=0$.
	\item $\mu_d(\Omega_m\setminus\Omega)\to0$.
\end{itemize}
Note that $\{\xx\in\Omega:\xx\not\in\Omega_m\mbox{ i.o.}\}=\bigcap_{M=1}^\infty\bigcup_{m=M}^\infty(\Omega\setminus\Omega_m)=\limsup_{m\to\infty}(\Omega\setminus\Omega_m)$. In particular, if $\mu_d(\Omega)<\infty$, then
\begin{align*}
\mu_d\{\xx\in\Omega:\xx\not\in\Omega_m\mbox{ i.o.}\}&=\mu_d\Biggl(\bigcap_{M=1}^\infty\bigcup_{m=M}^\infty(\Omega\setminus\Omega_m)\Biggr)=\lim_{M\to\infty}\mu_d\Biggl(\bigcup_{m=M}^\infty(\Omega\setminus\Omega_m)\Biggr)\ge\limsup_{M\to\infty}\mu_d(\Omega\setminus\Omega_M),
\end{align*}
and the first condition implies that $\mu_d(\Omega\setminus\Omega_m)\to0$.
We remark that the first condition can be rewritten as $\mu_d\{\xx\in\Omega:\xx\in\Omega_m\mbox{ eventually}\}=\mu_d(\Omega)$ in the case where $\mu_d(\Omega)<\infty$.

\begin{lemma}\label{auxiliary_lemma}
For every $m\in\mathbb N$, let $\{\xi(m,\nn)\}_{\nn\in\mathbb N^d}$ be a family of numbers such that $\xi(m,\nn)\to 0$ as $\nn\to\infty$.
Then, there exists a family $\{m(\nn)\}_{\nn\in\mathbb N^d}\subseteq\mathbb N$ such that $m(\nn)\to\infty$ and $\xi(m(\nn),\nn)\to0$ as $\nn\to\infty$.
\end{lemma}

\begin{lemma}\label{CH}
Let $f:\Omega\subset\mathbb R^d\to\mathbb R$ be measurable on a set $\Omega$ with $0<\mu_d(\Omega)<\infty$, and let $f_m:\Omega_m\subset\mathbb R^d\to\mathbb R$ be measurable on a set $\Omega_m$ with $0<\mu_d(\Omega_m)<\infty$. Suppose that $\Omega_m\to\Omega$ and $f_m(\xx)\to f(\xx)$ for a.e.\ $\xx\in\Omega$. Then, 
\begin{enumerate}[nolistsep,leftmargin=*]
	\item $F_{f_m}(u)\to F_f(u)$ for every continuity point $u$ of $F_f$,
	\item $f_m^\dag(y)\to f^\dag(y)$ for every continuity point $y$ of $f^\dag$.
\end{enumerate}
\end{lemma}
\begin{proof}\hfill
\begin{enumerate}[nolistsep,leftmargin=*]
\item Let $u\in\mathbb R$ be a continuity point for the distribution function $F_f$ and let $\epsilon>0$. Then, there exists $\delta=\delta_{u,\epsilon}>0$ such that
\[ F_f(u+\delta)-F_f(u-\delta)\le\epsilon. \]
Since $\Omega_m\to\Omega$ and $f_m(\xx)\to f(\xx)$ for a.e.\ $\xx\in\Omega$, we have
\[ \lim_{m\to\infty}\mu_d\{\xx\in\Omega\cap\Omega_m:|f(\xx)-f_m(\xx)|\ge\delta\}=0. \]
Indeed, if $\tilde f_m$ is any extension of $f_m$ to $\mathbb R^d$, then $\tilde f_m(\xx)\to f(\xx)$ for a.e.\ $\xx\in\Omega$ and
\begin{align*}
\mu_d\{\xx\in\Omega\cap\Omega_m:|f(\xx)-f_m(\xx)|\ge\delta\}&\le\mu_d\{\xx\in\Omega:|f(\xx)-\tilde f_m(\xx)|\ge\delta\}\to0.
\end{align*}
As a consequence, there exists $M=M_\delta$ such that, for $m\ge M$,
\[ \mu_d\{\xx\in\Omega\cap\Omega_m:|f(\xx)-f_m(\xx)|\ge\delta\}+\mu_d(\Omega\setminus\Omega_m)+\mu_d(\Omega_m\setminus\Omega)\le\epsilon. \]
Note that
\begin{align*}
\{\xx\in\Omega\cap\Omega_m:f_m(\xx)\le u\}&\subseteq\{\xx\in\Omega\cap\Omega_m:f(\xx)-|f(\xx)-f_m(\xx)|\le u\}\\
&=\{\xx\in\Omega\cap\Omega_m:|f(\xx)-f_m(\xx)|<\delta,\ f(\xx)-|f(\xx)-f_m(\xx)|\le u\}\\
&\hphantom{=\{}\,\cup\{\xx\in\Omega\cap\Omega_m:|f(\xx)-f_m(\xx)|\ge\delta,\ f(\xx)-|f(\xx)-f_m(\xx)|\le u\}\\
&\subseteq\{\xx\in\Omega\cap\Omega_m:f(\xx)\le u+\delta\}\cup\{\xx\in\Omega\cap\Omega_m:|f(\xx)-f_m(\xx)|\ge\delta\}
\end{align*}
and similarly
\[ \{\xx\in\Omega\cap\Omega_m:f(\xx)\le u-\delta\}\subseteq\{\xx\in\Omega\cap\Omega_m:f_m(\xx)\le u\}\cup\{\xx\in\Omega\cap\Omega_m:|f(\xx)-f_m(\xx)|\ge\delta\}. \]
Thus, for $m\ge M$,
\begin{align*}
F_{f_m}(u)&=\frac{\mu_d\{\xx\in\Omega_m:f_m(\xx)\le u\}}{\mu_d(\Omega_m)}\le\frac{\mu_d\{\xx\in\Omega_m\cap\Omega:f_m(\xx)\le u\}+\mu_d(\Omega_m\setminus\Omega)}{\mu_d(\Omega_m)}\\
&\le\frac{\mu_d\{\xx\in\Omega_m\cap\Omega:f(\xx)\le u+\delta\}+\mu_d\{\xx\in\Omega_m\cap\Omega:|f(\xx)-f_m(\xx)|\ge\delta\}+\mu_d(\Omega_m\setminus\Omega)}{\mu_d(\Omega_m)}\\
&\le F_f(u+\delta)\frac{\mu_d(\Omega)}{\mu_d(\Omega_m)}+\frac{\epsilon}{\mu_d(\Omega_m)}\le(F_f(u-\delta)+\epsilon)\frac{\mu_d(\Omega)}{\mu_d(\Omega)-\mu_d(\Omega\setminus\Omega_m)}+\frac{\epsilon}{\mu_d(\Omega)-\mu_d(\Omega\setminus\Omega_m)}\\
&\le(F_f(u)+\epsilon)\frac{\mu_d(\Omega)}{\mu_d(\Omega)-\epsilon}+\frac{\epsilon}{\mu_d(\Omega)-\epsilon}
\end{align*}
and
\begin{align*}
F_f(u)&\le F_f(u+\delta)\le F_f(u-\delta)+\epsilon=\frac{\mu_d\{\xx\in\Omega:f(\xx)\le u-\delta\}}{\mu_d(\Omega)}+\epsilon\\
&\le\frac{\mu_d\{\xx\in\Omega\cap\Omega_m:f(\xx)\le u-\delta\}+\mu_d(\Omega\setminus\Omega_m)}{\mu_d(\Omega)}+\epsilon\\
&\le\frac{\mu_d\{\xx\in\Omega\cap\Omega_m:f_m(\xx)\le u\}+\mu_d\{\xx\in\Omega\cap\Omega_m:|f(\xx)-f_m(\xx)|\ge\delta\}+\mu_d(\Omega\setminus\Omega_m)}{\mu_d(\Omega)}+\epsilon\\
&\le F_{f_m}(u)\frac{\mu_d(\Omega_m)}{\mu_d(\Omega)}+\frac{\epsilon}{\mu_d(\Omega)}+\epsilon\le F_{f_m}(u)\frac{\mu_d(\Omega)+\mu_d(\Omega_m\setminus\Omega)}{\mu_d(\Omega)}+\frac{\epsilon}{\mu_d(\Omega)}+\epsilon\\
&\le F_{f_m}(u)\frac{\mu_d(\Omega)+\epsilon}{\mu_d(\Omega)}+\frac{\epsilon}{\mu_d(\Omega)}+\epsilon.
\end{align*}
In conclusion, for $m\ge M$ we can write $\varphi(\epsilon)\le F_{f_m}(u)\le\psi(\epsilon)$, where both $\varphi(\epsilon)$ and $\psi(\epsilon)$ tend to $F_f(u)$ as $\epsilon\to0$. This proves that $F_{f_m}(u)\to F_f(u)$.

\item Let $y\in(0,1)$ be a continuity point for $f^\dag$. Since $F_f$ is monotone, its continuity points are dense in its domain~$\mathbb R$. As a consequence, $u=f^\dag(y)-\delta$ is a continuity point of $F_f$ for arbitrarily small $\delta>0$. From the definition of $f^\dag(y)$,
\[ F_f(f^\dag(y)-\delta)<y. \]
By the first statement of the lemma, $F_{f_m}(u)\to F_f(u)$ and so we eventually have
\[ F_{f_m}(f^\dag(y)-\delta)<y\quad\implies\quad f^\dag_m(y)\ge f^\dag(y)-\delta. \]
Thus,
\[ \liminf_{m\to\infty}f^\dag_m(y)\ge f^\dag(y)-\delta. \]
Since we can take $\delta$ arbitrarily small, we obtain
\[ \liminf_{m\to\infty}f^\dag_m(y)\ge f^\dag(y). \]
Now, fix $\epsilon>0$ and let $\delta>0$ such that $v=f^\dag(y+\epsilon)+\delta$ is a continuity point for $F_f$.
From the definition of $f^\dag(y+\epsilon)$ and the monotonicity of $F_f$,
\[ F_f(f^\dag(y+\epsilon)+\delta)\ge y+\epsilon. \]
By the first statement of the lemma, $F_{f_m}(v)\to F_f(v)$ and so we eventually have
\[ F_{f_m}(f^\dag(y+\epsilon)+\delta)\ge y\quad\implies\quad f^\dag_m(y)\le f^\dag(y+\epsilon)+\delta. \]
Thus,
\[ \limsup_{m\to\infty}f^\dag_m(y)\le f^\dag(y+\epsilon)+\delta. \]
Taking first the limit as $\delta\to0$ and then the limit as $\epsilon\to0$, by the continuity of $f^\dag$ in $y$ we obtain
\[ \limsup_{m\to\infty}f^\dag_m(y)\le f^\dag(y). \]
We then conclude that $f^\dag_m(y)\to f^\dag(y)$. \qedhere
\end{enumerate}
\end{proof}


\begin{proof}[Proof of Theorem~{\rm\ref{th:main2}}]
By Lusin's theorem \cite[Theorem~2.24]{Rudinone}, for every $m\in\mathbb N$ we can find a continuous function $f_m:[\aa,\bb]\to\mathbb R$ such that
\[ \mu_d(E_m)\le 2^{-m},\qquad E_m=\{\xx\in\Omega:f_m(\xx)\ne f(\xx)\}. \]
Since $\sum_{m=1}^\infty2^{-m}<\infty$, from the Borel--Cantelli lemma \cite[Theorem~1.41]{Rudinone} we infer that $f_m(\xx)$ is eventually equal to $f(\xx)$ for a.e.\ $\xx\in\Omega$, say for all $\xx\in\hat\Omega$ with $\mu_d(\hat\Omega)=\mu_d(\Omega)$. 

Now, fix $m\in\mathbb N$ and let
\[ \Omega_m=\Omega\setminus E_m=\{\xx\in\Omega:f_m(\xx)=f(\xx)\}. \]
For every $\nn\in\mathbb N^d$, choose a standard partition $\{I_{\ii,\nn}\}_{\ii=\bu,\ldots,\nn}$ of $[\aa,\bb]$
and define the grid $\mathcal G_{m,\nn}=\{\xx_{\ii,\nn}^{(m)}\}_{\ii=\bu,\ldots,\nn}$, where $\xx_{\ii,\nn}^{(m)}$ is a point in $I_{\ii,\nn}$ selected so that the following properties are satisfied:
\begin{equation*}
\left\{\begin{aligned}
&\xx_{\ii,\nn}^{(m)}\in\Omega_m, &&\mbox{if }\Omega_m\cap I_{\ii,\nn}\ne\emptyset,\\
&\xx_{\ii,\nn}^{(m)}\not\in E_m, &&\mbox{if }\Omega_m\cap I_{\ii,\nn}=\emptyset\ \wedge\ I_{\ii,\nn}\not\subseteq E_m. 
\end{aligned}\right.
\end{equation*}%
\begin{figure}
\centering
\includegraphics[width=0.45\textwidth]{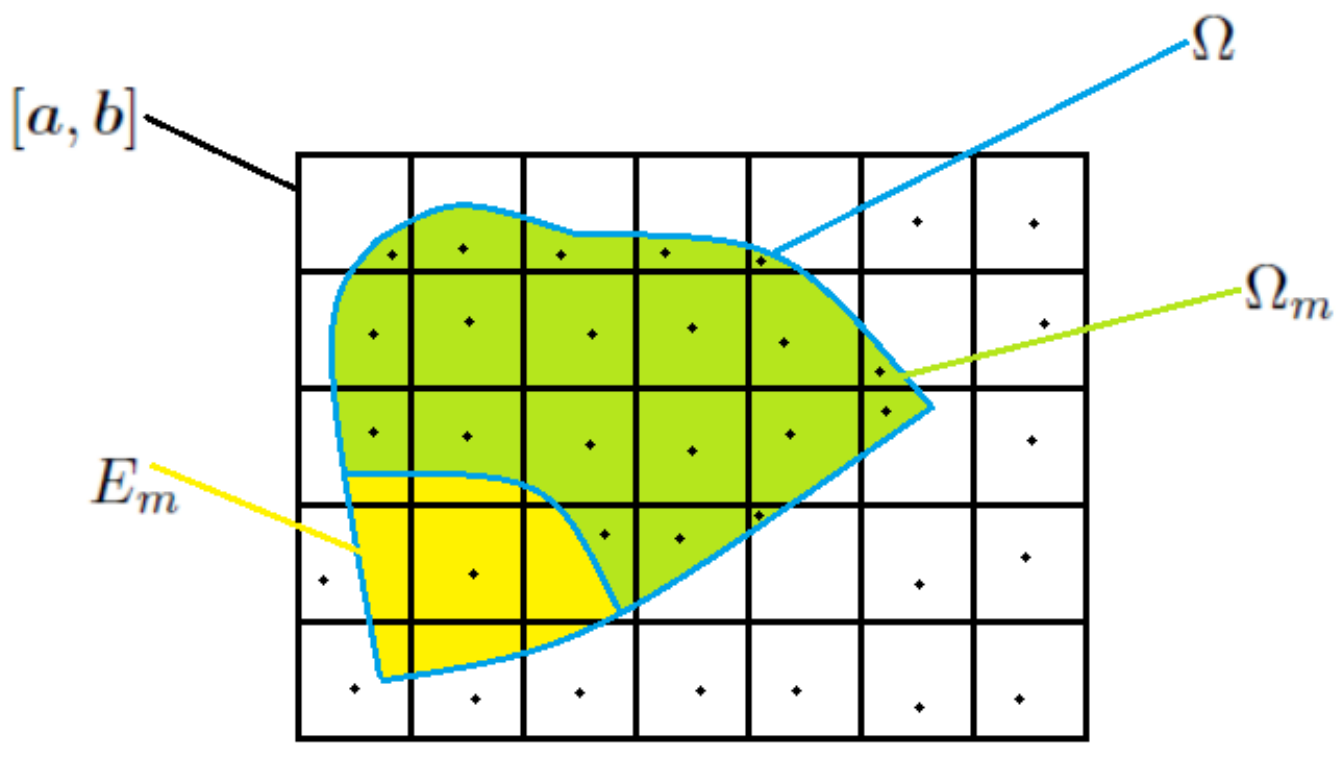}
\caption{Illustration for the proof of Theorem~\ref{th:main2}.}
\label{grid_choice}
\end{figure}%
In other words, whenever possible we take the grid point $\xx_{\ii,\nn}^{(m)}$ inside $\Omega_m$, and if it is not possible then we take $\xx_{\ii,\nn}^{(m)}$ outside $\Omega$ (provided that $I_{\ii,\nn}\not\subseteq\Omega$); see Figure~\ref{grid_choice}. In this way, $\xx_{\ii,\nn}^{(m)}\in E_m$ if and only if $I_{\ii,\nn}$ is entirely contained in $E_m$. Note that the grid $\mathcal G_{m,\nn}$ is a.u.\ in $[\aa,\bb]$ (regardless of $m$) because $\xx_{\ii,\nn}^{(m)}\in I_{\ii,\nn}$.

Define
\[ s_{m,\nn}(\xx)=\sum_{\ii=\bu}^\nn f_m(\xx_{\ii,\nn}^{(m)})\chi_{I_{\ii,\nn}}(\xx),\qquad\xx\in[\aa,\bb]. \]
It is not difficult to see that $s_{m,\nn}\to f_m$ uniformly on $[\aa,\bb]$ as $\nn\to\infty$. Indeed, for every $\xx\in[\aa,\bb]$, if we denote by $I_{\kk,\nn}$ the unique element of the standard partition containing $\xx$ and by $\omega_{f_m}$ the modulus of continuity of $f_m$ on $[\aa,\bb]$, then
\[ |s_{m,\nn}(\xx)-f_m(\xx)|=|f_m(\xx_{\kk,\nn}^{(m)})-f_m(\xx)|\le\omega_{f_m}\biggl(\biggl\|\frac{\bb-\aa}{\nn}\biggr\|_\infty\biggr), \]
which is independent of $\xx$ and tends to $0$ as $\nn\to\infty$ due to the continuity of $f_m$ on $[\aa,\bb]$. As a consequence, by Lemma~\ref{auxiliary_lemma} applied with $\xi(m,\nn)=\|s_{m,\nn}-f_m\|_{\infty,[\aa,\bb]}$, there exists a family $\{m(\nn)\}_{\nn\in\mathbb N^d}\subseteq\mathbb N$ such that $m(\nn)\to\infty$ and $\|s_\nn-f_{m(\nn)}\|_{\infty,[\aa,\bb]}\to0$ as $\nn\to\infty$, where $s_\nn=s_{m(\nn),\nn}$. This is enough to conclude that 
\[ s_\nn(\xx)=\sum_{\ii=\bu}^\nn f_{m(\nn)}(\xx_{\ii,\nn})\chi_{I_{\ii,\nn}}(\xx)\to f(\xx) \]
for every $\xx\in\hat\Omega$, where $\xx_{\ii,\nn}=\xx_{\ii,\nn}^{(m(\nn))}$. Indeed, for every $\xx\in\hat\Omega$, we have
\begin{align*}
|s_\nn(\xx)-f(\xx)|&\le|s_\nn(\xx)-f_{m(\nn)}(\xx)|+|f_{m(\nn)}(\xx)-f(\xx)|\le\|s_\nn-f_{m(\nn)}\|_{\infty,[\aa,\bb]}+|f_{m(\nn)}(\xx)-f(\xx)|,
\end{align*}
where the first term in the right-hand side tends to $0$ as $\nn\to\infty$ and the second term in the right-hand side is eventually equal to $0$ as $\nn\to\infty$.

Now, define $\mathcal G_\nn=\{\xx_{\ii,\nn}\}_{\ii=\bu,\ldots,\nn}$ and note that $\mathcal G_\nn$ is a.u.\ in $[\aa,\bb]$. We show that the thesis of the theorem holds for the grid $\mathcal G_\nn$.
Let $\II_\nn(\Omega)=\{\ii\in\{\bu,\ldots,\nn\}:\xx_{\ii,\nn}\in\Omega\}$ as in the statement of the theorem, and let
\[ t_\nn(\xx)=\sum_{\ii\in\II_\nn(\Omega)}f(\xx_{\ii,\nn})\chi_{I_{\ii,\nn}}(\xx),\qquad\xx\in Q_\nn=\bigcup_{\ii\in\II_\nn(\Omega)}I_{\ii,\nn}. \]
The domain $Q_\nn$ of $t_\nn$ contains $\Omega_{m(\nn)}$ by construction. If $\xx\in\hat\Omega$ then $f_{m(\nn)}(\xx)=f(\xx)$ eventually as $\nn\to\infty$ (i.e., $\xx\in\Omega_{m(\nn)}$ eventually as $\nn\to\infty$), hence
\[ \mu_d\{\xx\in\Omega:\xx\in Q_\nn\mbox{ eventually as }\nn\to\infty\}=\mu_d(\Omega). \]
Moreover, when $\ii$ varies in $\II_\nn(\Omega)$, the elements $I_{\ii,\nn}$ that are not fully contained in $\Omega$ contain anyway a point $\xx_{\ii,\nn}\in\Omega$ and so they necessarily intersect the boundary $\partial\Omega$ (e.g., in a point $\boldsymbol z$ of the segment connecting $\xx_{\ii,\nn}$ with a point $\boldsymbol y\in I_{\ii,\nn}\setminus\Omega$). In particular, the union of the elements $I_{\ii,\nn}$, $\ii\in\II_\nn(\Omega)$, that are not fully contained in $\Omega$ is contained in the set
\[ \Gamma_\nn=\biggl\{\xx\in\mathbb R^d:{\rm dist}(\xx-\partial\Omega)\le\biggl\|\frac{\bb-\aa}{\nn}\biggr\|_\infty\biggr\}, \]
where ${\rm dist}(\xx,\partial\Omega)$ is the distance of $\xx$ from $\partial\Omega$ (in $\infty$-norm).
Since $\mu_d(\partial\Omega)=0$ (because $\Omega$ is regular), we obtain
\[ \mu_d(Q_\nn\setminus\Omega)\le \mu_d\Biggl(\bigcup_{\ii\in\II_\nn(\Omega)\,:\,I_{\ii,\nn}\not\subseteq\Omega}I_{\ii,\nn}\Biggr)\le \mu_d(\Gamma_\nn)\to \mu_d(\partial\Omega)=0\mbox{ \,as \,}\nn\to\infty. \]
We have therefore proved that $Q_\nn\to\Omega$. Now, if $\xx\in\hat\Omega$, we have $f_{m(\nn)}(\xx)=f(\xx)$ eventually as $\nn\to\infty$ and
\[ |t_\nn(\xx)-f(\xx)|\le|t_\nn(\xx)-s_\nn(\xx)|+|s_\nn(\xx)-f(\xx)|, \]
where the first term in the right-hand side is eventually equal to $0$ as $\nn\to\infty$ and the second term in the right-hand side tends to $0$ as $\nn\to\infty$. In conclusion,
\[ t_\nn(\xx)\to f(\xx) \]
for all $\xx\in\hat\Omega$, and so $t_\nn\to f$ for a.e.\ $\xx\in\Omega$. We can now apply Lemma~\ref{CH} to conclude that $t_\nn^\dag(y)\to f^\dag(y)$ for every continuity point $y$ of $f^\dag$.

Since $t_\nn^\dag$ is the monotone rearrangement of the simple function $t_\nn$ and all the elements $I_{\ii,\nn}$ have the same measure regardless of $\ii$, we infer that $t_\nn^\dag$ is again a (left-continuous) simple function. More precisely, $t_\nn^\dag$ is given by
\[ t_\nn^\dag(y) = \sum_{i=0}^{\omega(\nn)}s_i\,\chi_{\left(\frac{i}{\omega(\nn)^{\vphantom{1}}+1},\frac{i+1}{\omega(\nn)^{\vphantom{1}}+1}\right]}(y),\qquad y\in(0,1), \]
where $\omega(\nn)=\#\II_\nn(\Omega)-1$ and $s_0,\ldots,s_{\omega(\nn)}$ are the samples $f(\xx_{\ii,\nn})$, $\ii\in\II_\nn(\Omega)$, sorted in non-decreasing order, as in the statement of the theorem. Note that the linear spline function $f^\dag_\nn$ mentioned in the statement of the theorem is explicitly given by
\begin{align*}
\left\{\begin{aligned}
&f^\dag_\nn\Bigl(\frac i{\omega(\nn)}\Bigr)=s_i, &&i=0,\ldots,\omega(\nn),\\
&f^\dag_\nn\mbox{ linear on }\biggl[\frac i{\omega(\nn)},\frac{i+1}{\omega(\nn)}\biggr], &&i=0,\ldots,\omega(\nn)-1,
\end{aligned}\right.
\end{align*}
so $f_\nn^\dag$ is ``close'' to $t_\nn^\dag$. In particular, denoting again with $t_\nn^\dag$ the obvious extension of $t_\nn^\dag$ to $[0,1]$ obtained by setting $t_\nn^\dag(0)=s_0$ and $t_\nn^\dag(1)=s_{\omega(\nn)}$, and observing that $\frac i{\omega(\nn)+1}<\frac i{\omega(\nn)}<\frac{i+1}{\omega(\nn)+1}$ for all $i=1,\ldots,\omega(\nn)-1$, we have
\[ 
s_i=f_\nn^\dag\Bigl(\frac{i}{\omega(\nn)}\Bigr)=t_\nn^\dag\Bigl(\frac{i}{\omega(\nn)}\Bigr),\qquad i=0,\ldots,\omega(\nn). \]
Now, let $y\in(0,1)$ be a continuity point for $f^\dag$ and let $\epsilon>0$. 
Take $\delta=\delta_{y,\epsilon}>0$ such that 
\[ \left\{\begin{aligned}
&f^\dag(y+\delta) - f^\dag(y-\delta)\le\epsilon,\\
&y+\delta,\ y-\delta\in(0,1)\mbox{ \,are continuity points for \,}f^\dag.
\end{aligned}\right. \]
Note that such a $\delta$ exists because $y$ is a continuity point of $f^\dag$ and the set of discontinuity points of $f^\dag$ is countable as $f^\dag$ is monotone.
Since $\omega(\nn)\to\infty$ as $\nn\to\infty$, we eventually have
\[ y-\delta\le\frac{i_\nn}{\omega(\nn)}\le y\le\frac{i_\nn +1}{\omega(\nn)}\le y+\delta \]
for some index $i_\nn\in\{0,\ldots,\omega(\nn)-1\}$, and
\[ t_\nn^\dag(y-\delta)\le t_\nn^\dag\Bigl(\frac{i_\nn}{\omega(\nn)} \Bigr) = f_\nn^\dag \Bigl( \frac{i_\nn}{\omega(\nn)} \Bigr) \le f_\nn^\dag(y) \le f_\nn^\dag\Bigl( \frac{i_\nn +1}{\omega(\nn)} \Bigr) =t_\nn^\dag \Bigl( \frac{i_\nn +1}{\omega(\nn)} \Bigr) \le t_\nn^\dag (y+\delta). \]
Since $t_\nn^\dag(z)\to f^\dag(z)$ for every continuity point $z$ of $f^\dag$, we infer that $t_\nn^\dag(y-\delta)\to f^\dag(y-\delta)$ and $t_\nn^\dag(y+\delta)\to f^\dag(y+\delta)$. Hence,
\begin{align*}
f^\dag(y)-\epsilon&\le f^\dag(y+\delta)-\epsilon\le f^\dag(y-\delta) = \lim_{\nn\to\infty}	t_\nn^\dag (y-\delta)\le \liminf_{\nn\to\infty} f_\nn^\dag(y),\\
\limsup_{\nn\to\infty} f_\nn^\dag(y)&\le \lim_{\nn\to\infty}	t_\nn^\dag (y+\delta) = f^\dag(y+\delta) \le f^\dag(y-\delta)+\epsilon \le f^\dag(y) +\epsilon.
\end{align*}
This is true for every $\epsilon>0$ and so we conclude that $f_\nn^\dag(y)\to f^\dag(y)$.
\end{proof}

\section*{Acknowledgements}
The first and third authors are members of the Research Group GNCS (Gruppo Nazionale per il Calcolo Scientifico) of INdAM (Istituto Nazionale di Alta Matematica).
The second author is member of the Research Group GNAMPA (Gruppo Nazionale per l'Analisi Matematica, la Probabilit\`a e le loro Applicazioni) of INdAM. 
This work has been supported by the MIUR Excellence Department Project awarded to the Department of Mathematics of the University of Rome Tor Vergata (CUP E83C18000100006) and by the Beyond Borders Programme of the University of Rome Tor Vergata through the Project ASTRID (CUP E84I19002250005).

\end{document}